\documentclass[12pt,titlepage]{article}
\usepackage{amstext}
\usepackage{amsmath}
\usepackage{amsthm}
\usepackage{amssymb,amsfonts}
\usepackage[mathscr]{eucal}
\usepackage[latin1]{inputenc}
\usepackage{graphicx}
\usepackage{url}
\usepackage{xcolor}
\newtheorem{Def}{Definition}
\newtheorem{The}[Def]{Theorem}

\newtheorem{Rem}[Def]{Remark}
\newtheorem{Pro}[Def]{Proposition}

\newtheorem{Ex}[Def]{Example}
\newtheorem{Cex}[Def]{Counterexample}

\begin{document}

\title{A new stochastic dominance criterion for dependent random variables with applications}

\author{F\'elix Belzunce \\
    Department of Estad\'{\i}stica e Investigaci\'on Operativa, University of Murcia \\ Facultad de Matemáticas, Campus de Espinardo \\ 30100 Espinardo (Murcia) \\ SPAIN \\ e-mail: \texttt{belzunce@um.es} 
    \\ and \\
    Carolina Mart\'{i}nez-Riquelme \\
    Department of Estad\'{\i}stica e Investigaci\'on Operativa, University of Murcia \\ Facultad de Matemáticas, Campus de Espinardo \\ 30100 Espinardo (Murcia) \\ SPAIN \\ e-mail: \texttt{carolina.martinez7@um.es @um.es} }

\date{}

\maketitle

\begin{abstract} In this paper we develop a new tool for the comparison of paired data based on a new criterion of stochastic dominance that takes into account the dependence structure of the random variables under comparison. This new procedure provides a more detailed comparison of dependent random variables and overcomes some difficulties of standard techniques like Student's t and Wilcoxon-Mann-Whitney tests for non normal data. This tool provides an alternative to the usual stochastic dominance criterion which only considers the marginal distributions in the comparison.  We show how this new tool can be fruitfully used for the comparison of paired asset returns. 



\end{abstract}

\section{Introduction}

The comparison of random quantities according to their magnitude has a long history. The most simple and basic way is the comparison of measures of location like the mean or median, whereas the use of stochastic dominance criteria have been proved to provide more detailed comparisons, to such extent that it is a fundamental tool for the comparison of random quantities in economics, finance and risk analysis nowadays.

The usual stochastic dominance criterion was introduced by Lehmann (1955) and has been extensively applied in economics, finance and risk analysis as it can be seen in the books by Müller and Stoyan (2002), Denuit et al. (2005), Sriboonchitta et al. (2010) and Levy (2016). Some other applications can be found in Shaked and Shanthikumar (2007) and Belzunce et al. (2016a). This criterion is commonly used to rank distributions as an alternative to the comparisons provided by measures of location.  Given two random variables $X$ and $Y$, $X$ is said to be less or equal than $Y$ in the stochastic dominance criterion (or stochastic order), denoted by $X\le_{\textup{st}}Y$, if and only if, $P(X>x)\le P(Y>x)$, for all $x\in \mathbb R$.  Clearly, this criterion compares the magnitude of the two random variables, where the random variable $Y$ tends to take larger values than the random variable $X$ from a probabilistic point of view. In particular, if the stochastic dominance holds, then the means, the medians and every percentile or quantile of the corresponding random variables are also ordered. The number of statistical inference techniques about the stochastic dominance criterion have increased during the last two decades; see, for example, Davison and Duclos (2000), Barret and Donald (2003), Linton et al.  (2005) and more recently Scaillet and  Topaloglou (2010), Barret et al.  (2014) and Andreoli (2018). However, the main drawback of the stochastic dominance criterion is that random variables and does not take into account the dependence structure when the random variables are dependent, which is the case of returns for two assets.  Despite the fact that several criteria of stochastic dominance for dependent random variables have been considered (see for example Shanthikumar and Yao,  1991 and Cai and Wei,  2014), they are very technical in nature and are not easy to check from a statistical point of view. 

Therefore, the aim of this paper is to propose a new tool for the comparison of dependent random variables, which is easy to deal with,  according to a new stochastic dominance criterion that takes into account the dependence structure and to show how this tool can be used in practice to improve the outcome of a portfolio by selecting the assets according to the new criterion,  covering situations where the assets are not ordered in the usual stochastic dominance criterion and/or the dependence is not taken into account. 

In order to provide a motivation of the new tool,  we start making a review of the usual statistical techniques for the comparison of dependent random variables,  which are based on a paired sample from a bivariate random vector $(X,Y)$.  If the difference $Y-X$ is normally distributed,  the standard procedure is to apply the Student's t test for $E[Y-X]$, concluding that $Y$ tends to take larger values than $X$ if $E[Y-X]>0$. As we will see later, this property implies, in some probabilistic sense, that the random variable with the greatest mean truly tends to take larger values than the other one (see Example \ref{belip}) in contrast with the independent case. The problem arises when $Y-X$ is not normally distributed. In such case, the Wilcoxon-Mann-Whitney (WMW) test is commonly performed for the median of $Y-X$ in order to check if the median is greater than 0. However, the median approach has some problems and disadvantages that we consider next. From the statistical point of view, for instance, Divine et al. (2018) have demonstrated that the WMW procedure fails as a test to compare two medians. From the probabilistic point of view, let us pose the following question: Is it appropriate to consider the median of $Y-X$ as a measure to check whether a random variable tends to take larger values than another one, when $X$ and $Y$ are dependent? Next, we show that the answer is no, as well as we explore which criterion should be considered to compare the magnitude of two dependent random variables. 

Let us first show the reason why the comparison based on the median of $Y-X$ is not appropriate to compare dependent random variables. The idea behind the WMW test is to check if the median of $Y-X$ is greater than 0 (or, equivalently, if the median of $X-Y$ is smaller than 0) as an evidence that $Y$ tends to take larger values than $X$. If we assume that $P(X=Y)=0$, then the median of $Y-X$ is greater than  0 if, and only if, $P(Y-X >0) \ge (0.5 \ge) P(X-Y > 0)$, this property is equivalent to the strict stochastic precedence criterion and it is denoted by $X\le_{\textup{pr}}Y$ (see Boland et al., 2004 and a related definition by Arcones et al., 2002).  However, the following example shows that the stochastic precedence criterion is not very informative to compare the magnitude of two dependent random variables. 

\begin{Ex}\label{Ex1}
Let us consider the following example in which we have a discrete bivariate distribution on $\mathbb R^2$, such that all the points have the same probability  (see Figure \ref{fig1}). In this case, it is clear that $P(X>Y)=0.5=P(Y>X)$ and, therefore, the two random variables are equal according to the stochastic precedence criterion. However, if we look at the differences $Y-X$, when $Y$ is greater than $X$ (dashed lines), these are greater than the differences between $X-Y$, when $X$ is greater than $Y$ (dotted lines) and, intuitively, $Y$ tends to take larger values than $X$.
\begin{figure}
\centering
\includegraphics[scale=0.7]{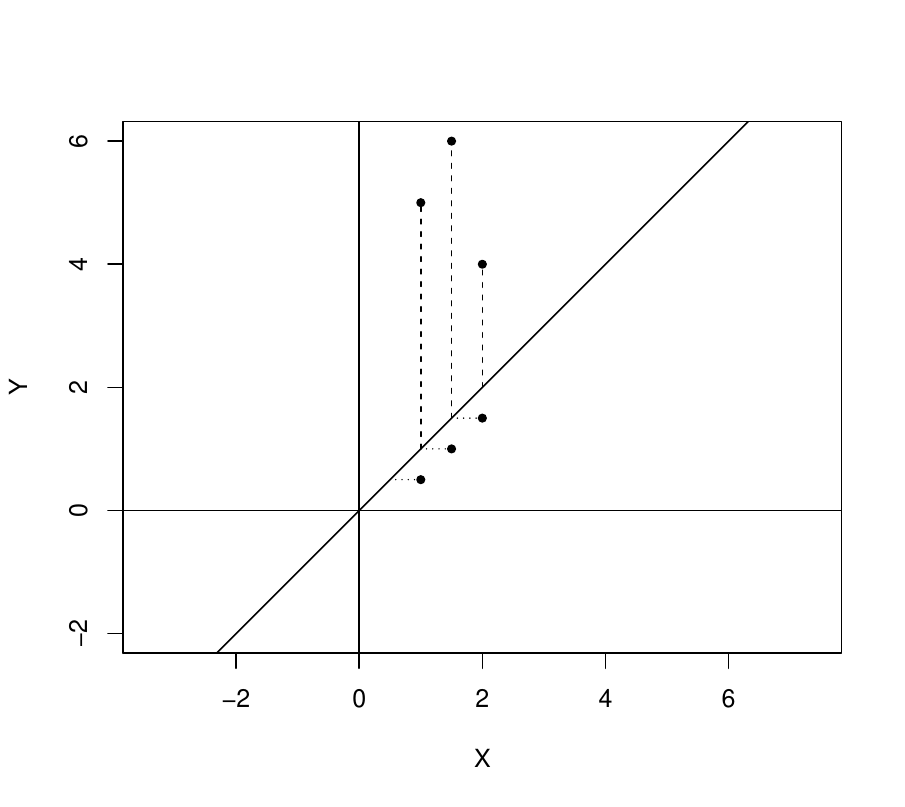}
\caption{\label{fig1}\textit{A discrete distribution on $\mathbb R^2$ where each point has the same probability equal to 1/6.}}
\end{figure} 
\end{Ex}

The problem lies in reducing all the information about $Y-X$ to just one number, that is, to $P(Y-X>0)$,  without taking into account the thickness of the left tail (below the median) and the right tail (above the median) of $Y-X$. Moreover, we observe that the stochastic dominance criterion is not appropriate in this case either, since it only compares the marginal distributions of the two random variables without taking into account their dependence structure and, therefore, it should not be considered as a criterion to compare two dependent random variables. Because of these facts, we wonder about which probabilistic criterion should be considered to compare the magnitude of two random variables in the dependent case in order to affirm that one of the random variables tends to take larger values than the other one. 

Going back to Example \ref{Ex1}, the stochastic precedence criterion does not take into account the magnitudes of $Y-X$ and $X-Y$, therefore we can consider the comparison of the magnitude of $Y-X$ and $X-Y$ to get a more detailed comparison of $X$ and $Y$. A natural way to compare such differences is to consider the usual stochastic dominance criterion, that is, to check if $X-Y\le_{st} Y-X$. Since the stochastic dominance criterion compares the magnitude of two random variables, if the previous inequality holds, then $Y-X$ tends to take large values than $X-Y$, and therefore, $Y$ tends to take larger values than $X$ taking into account their joint distribution to reach the conclusion. Based on this idea,  we  introduce a new statistical test for the comparison of paired data.  A natural scenario where the new criterion can be applied is the comparison of asset returns. The use of stochastic dominance criteria to compare two assets with returns $X$ and $Y$ has been shown to be a very powerful tool to improve and compare portfolios.  As we will see, the new tool can be used to improve the outcomes of a portfolio by replacing one asset by another one which stochastically dominates the previous one according to the new criterion. 

The organization of this paper is as follows. In Section 2, we introduce a new stochastic dominance criterion for the comparison of the magnitude of two dependent random variables, as well as we study some of its main properties. Along the revision of this paper, professor Tomasso Lando noticed to us that this criterion appears in Montes et al. (2020) under a different notation. Fortunately, there is no overlap with the contents of this paper and the main focus in the paper by Montes et al. (2020) is to show the connection among the usual stochastic order and the new one under different types of dependence. Additionally, they present the new criterion in the context of regret theory which can be used to provide an interpretation in the context of risk theory, as we will observe later in Section 2.  Next, in Section 3 we derive a statistical procedure for testing the new criterion. In Section 4, we use the test to provide a fruitful application in the context of finance. Finally, in Section 5 we give some comments about open problems and future research about the new  stochastic dominance criterion. 

\section{Stochastic dominance among two dependent random variables}

According to the previous section, we introduce the following new stochastic dominance criterion as a tool for comparing two dependent random variables in terms of their magnitude.

\begin{Def}\label{wjhr} Given a bivariate random vector $(X,Y)$, we say that $X$ is smaller than $Y$ in the weak joint stochastic dominance, denoted by $X \le_{\textup{st:wj}} Y$, if
\[
P(X-Y>t)\le P(Y-X>t),\text{ for all }t\in \mathbb R,
\]
or, equivalently, if $X-Y \le_{\textup{st}}Y-X$.
\end{Def}

According to the preceding notation in the literature, we have used the term ``joint''  to indicate that we are considering the dependence structure in the comparison and the term ``weak'' to point out that it is a weaker criterion than an existing criterion, as we will see later. Prior to explore the properties of this order and the relationships with some other criteria found in the literature, let us make some remarks. 

This notion has been independently introduced by Montes et al. (2020), where they give an interpretation in the context of regret theory. This theory was introduced by Bell (1982), Fishburn (1982) and Loomes and Sugden (1982), where the satisfaction of a decision maker is based on the benefits of his/her choice and on the benefits of the unselected one. In particular, Montes et al. (2020) show that $X \le_{\textup{st:wj}} Y$ if,  and only if, $X \le_{\textup{AR}} Y$ (absolute regret order), where $X \le_{\textup{AR}} Y$ if, and only if, $E[u(X-Y)]\le E[u(Y-X)]$, for very increasing function $u$ (see Hon Tan and Hartman, 2013). Replacing ``benefits" by ``losses", the st:wj can be interpreted in the context of risk theory as follows. A decision maker would prefer risk $X$ instead of risk $Y$, taking into account the losses of the selected risk and the unselected one, if $X \le_{\textup{st:wj}} Y$.

\begin{Rem}\label{2.4} As observed by Montes et al (2020) the new criterion is coherent with both the comparison of the means of the two random variables and the comparison in the stochastic precedence criterion. It is easy to see that if $X\le_{\textup{st:wj}}Y$, then $E[X]\le E[Y]$. In addition, from the definition of the new criterion, it also follows that $P(X>Y)\le P(Y>X)$, that is, $X\le_{\textup{pr}}Y$.

Finally, notice that if $X\le_{\textup{st:wj}}Y$ then $E[X]=E[Y]$ if, and only if, $X-Y=_{\textup{st}}Y-X$ (see Theorem 1.A.8 in Shaked and Shanthikumar, 2007).
\end{Rem}

\begin{Rem}\label{2.1}  We observe that the new criterion compares the probability of the areas $A_t=\{(x,y)\in \mathbb R^2|y > x + t\}$ and $B_t=\{(x,y)\in \mathbb R^2| y  < x -t\}$, for all $t\in \mathbb R$, that is, the new criterion is equivalent to $P((X,Y)\in B_t) \le P((X,Y)\in A_t)$, for all $t \in \mathbb R$. It is easy to see that the previous inequality holds for the random variables $X$ and $Y$ defined in Example \ref{Ex1}, that is, $X \le_{\textup{st:wj}} Y$ in that case.
\end{Rem}

\begin{Rem}\label{2.2} If $X-Y$ (or equivalently $Y-X$) has a continuous distribution, then the weak joint stochastic dominance holds if, and only if, $P(X-Y>t)\le P(Y-X>t),\text{ for all }t\ge 0$. The proof follows trivially from the fact that $P(X-Y>t)=P(Y-X\le -t)$ and $P(Y-X>t)=P(X-Y\le -t)$, for all $t<0$. In other words, the right tail of $X-Y$ behaves like the left tail of $Y-X$ and, therefore, it is natural to study if, under the condition $X \le_{\textup{st:wj}} Y$,  the right tail of $Y-X$ is lighter than the left tail of  $Y-X$. The usual stochastic order has been related to the comparison of tails by Mulero et al. (2017). They prove that for two non negative random variables $X$ and $Y$, if $X\le_{\textup{st}}Y$, then $X\le_{\textup{rtail}}Y$, where $X\le_{\textup{rtail}}Y$ if 
\[
E[Xw(X)]\le E[Yw(Y)],
\]
for all non-decreasing function $w:\mathbb R \Rightarrow \mathbb R^+$ provided previous expectations exist. The rtail ordering has been proved to be consistent with the fact that the right tail of $X$ is heavier than the right tail of $Y$. Now, given that under the assumption $X \le_{\textup{st:wj}} Y$ we have that $(X-Y)^+ \le_{\textup{st}}(Y-X)^+$, then $(X-Y)^+ \le_{\textup{rtail}}(Y-X)^+$ and therefore the right tail of $Y-X$ is lighter than the right tail of $X-Y$, or equivalently, the left tail of $Y-X$.
\end{Rem}

Let us see now an example of a family of parametric distributions where the new  weak joint stochastic dominance holds.

\begin{Ex}\label{belip} Let us consider a bivariate elliptically distributed random vector, $(X,Y)\sim E_2(\mathbf{\mu}, \mathbf{\Sigma}, \Psi)$, that is, $(X,Y)^t=\mathbf{\mu} +\mathbf A \mathbf U$, where $\mathbf{\mu}$ is the vector of means,  $\mathbf A\in \mathbb R ^{2\times k}$ is a matrix of parameters such that $\mathbf{\Sigma} = \mathbf{AA}^t$,  and $\mathbf U$ is a $k$-dimensional spherical distribution with characteristic generator $\Psi$. From the well known properties of linear combinations of elliptical distributions, it is easy to see that $Y-X\sim E_1(\mu_2-\mu_1, s_{11} + s_{22} -2 s_{12}, \Psi)$ and $X-Y\sim E_1(\mu_1-\mu_2, s_{11} + s_{22} -2 s_{12}, \Psi)$, where 
\[
\mathbf{\Sigma}=\left ( 
\begin{matrix}
    s_{11} & s_{12} \\
    s_{12} & s_{22} 
  \end{matrix}
  \right ).
\]
That is, $Y-X$ and $X-Y$ are equally distributed, except for a location change. Therefore, if $E[X]=\mu_1\le E[Y]=\mu_2$, then $X-Y\le_{\textup{st}}Y-X$ or, equivalently, $X\le_{\textup{st:wj}}Y$.

Therefore, the comparison in the new criterion only depends on the position on the plane of the mean vector and does not depend on the variances and correlation among the two random variables. 

As a particular case, if we consider a bivariate normal random vector $(X,Y)\sim N(\mathbf \mu,\mathbf V)$, where $\mathbf \mu$ is the mean vector and $\mathbf V$ is the covariance matrix, if we assume that $E[X]=\mu_1 \le \mu_2=E[Y]$, then we have that $X\le_{\textup{st:wj}}Y$.  It is worthy mentioning that the previous comment is still true when the difference $X-Y$ is normally distributed, without any assumption on the joint distribution. As a consequence, the Student's t test for paired samples can be also considered as a test for the comparison of the magnitude of two dependent random variables according to the new criterion. 

Figure \ref{fig3} shows examples for both positive and negative correlation. In particular, we consider $(X,Y)\sim N(\mu,V)$, where $\mu=(1,2)$, the variances are 2 and 1, respectively, and the correlation is 0.8, for the positive correlation and  -0.8, for the negative correlation. The dashed lines are the boundaries of the areas $A_t$ and $B_t$, for $t=1.2$, whose probabilities can be compared to check the weak joint stochastic order (see Remark \ref{2.1}).
\end{Ex}

\begin{figure}[ht]\centering
\includegraphics[width=\textwidth]{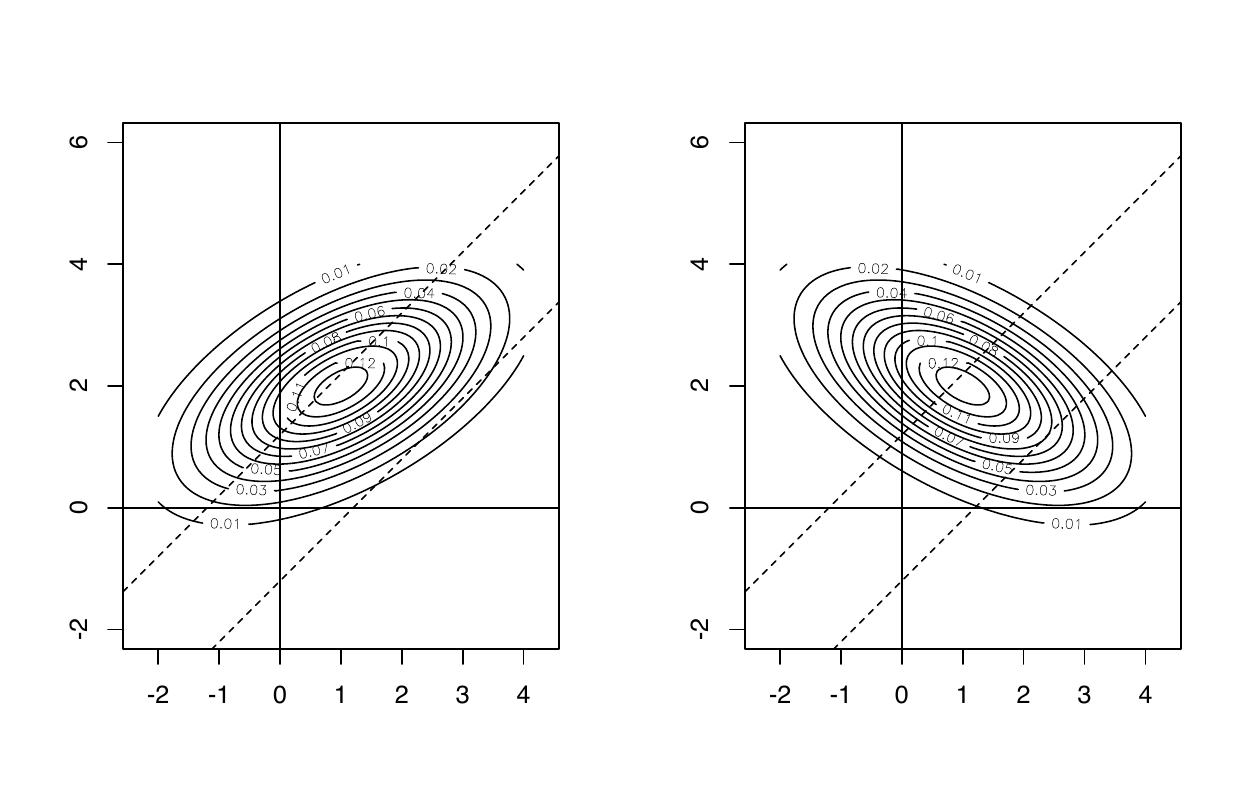}
\caption{\label{fig3}\textit{Joint density functions for $(X,Y)\sim N(\mu,V)$, where $\mu=(1,2)$, the variances are 2 and 1, the correlation is 0.8 (on the left) and  -0.8 (on the right), where $t=1.2$.}}
\end{figure}

Next, we explore the relationships between the new criterion and some other joint stochastic dominance criteria previously defined in the literature. As mentioned in the introduction, the comparison of two dependent random variables according to their magnitude was considered first by  Shanthikumar and Yao (1991) and more recently by Cai and Wei (2014) and Belzunce et al. (2016b), under two different approaches. Let us describe the two approaches.

The approach given by Shanthikumar and Yao (1991) is based on a preliminary characterization of the corresponding stochastic dominance criterion in the independent case,  considering the same characterization as a definition in the dependent case. For example, Shanthikumar and Yao (1991) proved the following characterization of the usual stochastic dominance criterion. 

Let $X$ and $Y$ be independent random variables, then $X\le_{\textup{st}}Y$ if, and only if, $E[g(X,Y)]$ $\le E[g(Y,X)]$, for all $g \in G_{st}$, provided that the previous expectations exist, where 
\[
G_{st}=\{g:\mathbb R^2 \rightarrow \mathbb R | g(x,y)-g(y,x)\text{ is increasing in }x,\text{ for all }y\}.
\]

According to this characterization, they provided the following definition. Given a bivariate random vector $(X,Y)$, $X$ is said to be smaller or equal than $Y$ in the joint stochastic order, denoted by $X\le_{\textup{st:j}}Y$, if $E[g(X,Y)]\le E[g(Y,X)],\text{ for all }g \in G_{st}$, provided that the previous expectations exist. Following this idea they proposed some other joint stochastic dominance criteria. The motivation of these criteria is purely mathematical based, and it is hard to find a motivation from the applied point of view.  Additionally,  there are no statistical tools to check these criteria. According to these comments, it is clear that a new criterion is needed and our proposal provides a simple answer to this problem.

Next, we show that the new criterion is weaker than the joint stochastic order. 

\begin{The} Given a bivariate random vector $(X,Y)$, if $X\le_{\textup{st:j}}Y$, then $X\le_{\textup{st:wj}}Y$.
\end{The}

\begin{proof} According to Shanhtikumar and Yao (1991), if $X\le_{\textup{st:j}}Y$, then $(X,-Y)\leq_{\textup{st}} (Y,-X)$ which, from the definition of the multivariate stochastic dominance (see Section 5), implies that $X-Y\leq_{\textup{st}} Y-X$, or equivalently, $X \le_{\textup{st:wj}} Y$.\hfill \end{proof}

Therefore, taking into account the previous results, comments and the results in Shanthikumar and Yao (1991),  we have the following chain of implications. 

\begin{center}
\begin{tabular}{c c c c c }
&  & $X\le_{\textup{st:j}}Y$ & $\Rightarrow$ & $X\le_{\textup{st}}Y$ \\
&  & $\Downarrow$ & & $\Downarrow$ \\
& $X\le_{\textup{pr}}Y$ $\Leftarrow$ & $X\le_{\textup{st:wj}}Y$ & $\Rightarrow$ & $E[X] \le E[Y]$.\\
\end{tabular}
\end{center}

Therefore, the new criterion is weaker than the previous criterion, which means that it has a wider applicability and it can only be improved by the comparison of medians and means, which is not very informative as we have already pointed out. 

Next, we provide some counterexamples to show that the implications related to the new criterion do not hold in the reverse sense. On the one hand, Example \ref{Ex1} can be used as a counterexample for the implication $X\le_{\textup{st:wj}}Y \Rightarrow X\le_{\textup{st}}Y$. On the other hand, the following counterexample shows that the implication $X\le_{\textup{st:wj}}Y \Rightarrow X\le_{\textup{st:j}}Y$ does not hold.

\begin{Cex}\label{Cx1} Let us consider $(X,Y)\sim N(\mathbf \mu, \mathbf V)$, where $E[X]=\mu_1 \neq \mu_2=E[Y]$ and $Var[X]=\mathbf V_{11} \neq \mathbf V_{22}=Var[Y]$. It is known that  $X\nleq_{\textup{st}}Y$ and $X\ngeq_{\textup{st}}Y$ (see Table 2.2 in Belzunce et al., 2016a). However, if $E[X]=\mu_1 < \mu_2=E[Y]$ we have that $X\le_{\textup{st:wj}}Y$. Since $X\le_{\textup{st:j}}Y \Rightarrow X\le_{\textup{st}}Y$, this example contradicts the implication $X\le_{\textup{st:wj}}Y \Rightarrow X\le_{\textup{st:j}}Y$.
\end{Cex} 

The main focus of the paper by Montes et al. (2020) is the relationship among the new order and the usual stochastic order. For example they prove that if  $X$ and $Y$ are independent random variables such that $X\le_{\textup{st}}Y$, then $X\le_{\textup{st:wj}}Y$. A simple proof of this fact is the following. If $X\le_{\textup{st}}Y$ then follows that $-Y\le_{\textup{st}}-X$. Now, by the preservation under convolution of the stochastic dominance, it holds that $X-Y\le_{\textup{st}}Y-X$ or, equivalently, $X\le_{\textup{st:wj}}Y$. Additionally, Montes et al. (2020) provide a counterexample  for the reversed implication.

Montes et al. (2020) also provide two results where the usual stochastic order implies the st:wj order under the assumption of comonotonicity or Archimedean copula of the random vector $(X,Y)$.

Next we provide a more general result for the case of Archimedean copulas. In particular, we provide a result for a general copula.

\begin{The}\label{Th-cop} Let $(X,Y)$ be a bivariate random vector with absolutely continuous distribution and  copula $C(p,q)$. If $X\le_{\textup{st}} Y$ and  
\begin{equation}\label{cond-cop}
\frac{\partial}{\partial p}C(u, v_1) \le \frac{\partial}{\partial q}C(v_2,u),\text{ for all }u,v_1 \text{ and } v_2\in(0,1),\text{ such that } v_1\le v_2.
\end{equation}
then $X\le_{\textup{st:wj}} Y$.
\end{The}

\begin{proof} From Theorem 2.3 by Navarro and Sarabia (2022) we have that the distribution function of $X-Y$ is given by
\[
P(X-Y\le t) = \int_{- \infty}^{+ \infty}\frac{\partial}{\partial p}C_{-}(F(x), \overline G(x-t))f(x)dx,
\]
where $C_{-}(p,q)=p-C(p, 1-q)$ is the copula of $(X,-Y)$ (see Theorem 2.4.4, Nelsen, 2006).

Let us consider in the previous integral the change of variable $F(x)=u$,  then we have 
\begin{equation}\label{par1}
P(X-Y\le t) = \int_0^1\left(1-\frac{\partial}{\partial p}C(u, G(F^{-1}(u)-t))\right)du.
\end{equation}

In a similar way we have that the distribution function of $Y-X$ is given by:
\begin{equation}\label{par2}
P(Y-X\le t) = \int_0^1\left(1-\frac{\partial}{\partial q}C(F(G^{-1}(u)-t),u)\right)du.
\end{equation}

Now given that $X\le_{\textup{st}}Y$ we have that $F(x)\ge G(x)$ for all $x \in \mathbb R$ and $F^{-1}(u) \le G^{-1}(u)$ for all $u \in (0,1)$, therefore
\[
F(G^{-1}(u)-t)\ge G(G^{-1}(u)-t) \ge G(F^{-1}(u)-t),\text{ for all }t\in\mathbb R,u\in(0,1).
\]

Now, from \eqref{cond-cop} we have that 
\[
\frac{\partial}{\partial p}C(u, G(F^{-1}(u)-t)) \le \frac{\partial}{\partial q}C(F(G^{-1}(u)-t),u),\text{ for all }t\in\mathbb R,u\in(0,1).
\]

Combining this inequality with equations \eqref{par1} and \eqref{par2} we get that $P(X-Y\le t) \ge P(Y-X\le t)$ for all $t\in \mathbb R$, that is, $X\le_{\textup{st:wj}}Y$.   \end{proof}

 If $C$ is an Archimedean copula, that is, $C(u,v)=\phi^{-1}(\phi(u) + \phi(v))$, where $\phi$ is differentiable an strictly decreasing and convex function, condition \eqref{cond-cop} is equivalent to 
 \[
 \phi'(\phi^{-1}(\phi(u)+\phi(v_1)))\le  \phi'(\phi^{-1}(\phi(u)+\phi(v_2))),u,v_1\le v_2
 \]
 
 Notice that previous condition holds from the decreasingness and convexity of function $\phi$.  Therefore, Theorem \ref{Th-cop} provides a generalization of Proposition 14 by Montes et al. (2020), where the generator is assumed to be a twice differentiable function.

To finish, we provide some desirable properties of a stochastic ordering such as stability under convolution, mixtures, compounding and limit (see Denuit et al. 2005, p. 106). First, we observe the preservation under convolution. 

\begin{The} Let $(X_1,Y_1),(X_2,Y_2),\ldots,(X_n,Y_n)$ be a set of independent bivariate random vectors. If $X_i\le_{\textup{st:wj}}Y_i$, for all $i=1,2,\ldots,n$, then 
\[
\sum_{i=1}^n X_i \le_{\textup{st:wj}}  \sum_{i=1}^n Y_i.
\]
\end{The}

\begin{proof} From the assumptions, it is clear that $X_i-Y_i \le_{\textup{st}}Y_i-X_i$, for all $i=1,\ldots,n$ and also that the random variables $X_1-Y_1,X_2-Y_2,\ldots,X_n-Y_n$
are independent, as well as the random variables $Y_1-X_1,Y_2-X_2,\ldots,Y_n-X_n$. Therefore, from the preservation of the stochastic dominance under convolutions, we have that 
\[
\sum_{i=1}^n (X_i - Y_i) \le_{\textup{st}}  \sum_{i=1}^n (Y_i - X_i),
\]
that is, the st:wj order holds among the convolutions of $X$'s and $Y$'s. \hfill \end{proof}

Now, we provide a result on mixtures. 

\begin{The}\label{theta} Let $\{(X(\theta),Y(\theta)), \theta \in S\subseteq \mathbb R \}$ be a family of bivariate random vectors. Let $\Theta_1$ and  $\Theta_2$ be two random variables with common support $S$. If 
\begin{enumerate}
\item $X(\theta) \le_{\textup{st:wj}} Y(\theta)$, for all $ \theta \in S$,
\item $X(\theta) - Y(\theta)$ or $Y(\theta) - X(\theta)$ are increasing in the usual stochastic dominance order in $\theta \in S$ and 
\item $\Theta_1 \le_{\textup{st}} \Theta_2$,
\end{enumerate}
then $X(\Theta_1) \le_{\textup{st:wj}} Y(\Theta_2)$. 
\end{The}

\begin{proof} Let us consider the random variables $Z(\mathbf \theta)=X(\theta) - Y(\theta)$ and $Z'(\mathbf \theta)=Y(\theta) - X(\theta)$. Then, from Assumption 1, we have that $Z(\mathbf \theta)\le_{\textup{st}} Z'(\mathbf \theta)$, for all $\theta \in S$. From Assumption 2, we have that $E[\phi(Z(\theta))]$ or $E[\phi(Z'(\theta))]$ are increasing in $\theta \in S$, for any increasing function $\phi$. The result follows now from Theorem 2.2.8, in Belzunce et al. (2016a).
\end{proof}

This result is useful to provide comparisons in the presence of covariates. The result can be easily extended  to the case where $\mathbf \Theta_1$ and  $\mathbf \Theta_2$ are two $n$ dimensional random vectors with common support $S \subseteq \mathbb R^n$, and it can be used to provide examples of bivariate vectors which are ordered according to the weak joint stochastic dominance.

Moreover, it can be used, jointly with the preservation under convolution, to provide a preservation result under compounding. In particular the two previous results can be used to provide the following theorem.

\begin{The} Let $(X_1,Y_1),(X_2,Y_2),\ldots,(X_n,Y_n)$ be a set of independent bivariate random vectors. If $X_i\le_{\textup{st:wj}}Y_i$, for all $i=1,2,\ldots,n$, and $N$ and $M$ two integer-valued random variables independent of the random variables $(X_1,Y_1),(X_2,Y_2),\ldots,(X_n,Y_n)$. If $X_i\le_{\textup{st:wj}}Y_i$, for all $i=1,2,\ldots,n$ and $N\le_{\textup{st}}M$ then  
\[
\sum_{i=1}^N X_i \le_{\textup{st:wj}}  \sum_{i=1}^M Y_i.
\]
\end{The}

Finally, we show that the new criterion is preserved under convergence in distribution. 

\begin{The} Let $\{(X_n, Y_n)\}_{n\in \mathbb N}$ be a sequence of bivariate random vectors which converges in distribution to a bivariate random vector $(X,Y)$. If $X_n\le_{\textup{st:wj}}Y_n$, for all $n\in \mathbb N$, then $X\le_{\textup{st:wj}}Y$.
\end{The}

\begin{proof} If $\{(X_n, Y_n)\}_{n\in \mathbb N}$ converges in distribution to a bivariate random vector $(X,Y)$, then 
$X_n-Y_n$ converges in distribution to $X-Y$ and $Y_n-X_n$ converges in distribution to $Y-X$. The result follows for the preservation of the usual stochastic dominance under convergence (see Theorem 2.2.7 in Belzunce et al., 2016a). 
\end{proof}

It would be interesting to find an integral representation of this new stochastic order. In particular, it would be interesting to find a characterization in the sense of the definition of joint stochastic orders provided by Shanthikumar and Yao (1991) but we have not been able to provide it and it remains as an open problem.

\section{A non parametric asymptotic test for the weak joint stochastic dominance}

In this section, we develop a non parametric asymptotic test for testing the weak joint stochastic dominance criterion.  Given that the focus is on the comparison of two dependent random variables $X$ and $Y$, our purpose is to develop a test based on a paired random sample, $\left\{ (X_i,Y_i)\right\}_{i=1}^n$, of $(X,Y)$. We will assume that $Y-X$ (and, obviously, $X-Y$) has a continuous distribution function. A remark will be given at the end of this section for the cases of integer-valued and ordinal random variables.

Under these assumptions, the goal of this section is to provide a non parametric test for testing the null hypothesis 
\[
H_0: X\le_{\textup{st:wj}}Y 
\]
against the alternative hypothesis
\[
H_1: X \nleq_{\textup{st:wj}}Y,
\]
based on the paired sample $\left\{ (X_i,Y_i)\right\}_{i=1}^n$.  Taking into account Remark \ref{2.2}  we observe that the previous test is equivalent to test the null hypothesis 
\[
H_0: P(X-Y>t) \le P(Y-X>t),\text{ for all }t\in[0,+\infty) 
\]
against the alternative hypothesis
\[
H_1: P(X-Y>t) > P(Y-X>t),\text{ for some }t\in[0,+\infty).
\]

Following Barret and Donald (2003), we propose a Kolmogorov-Smirnov type test of significance. Prior to the definition of the test statistic, let us fix some notation and make some observations. Let us denote by $F(t)$ and $G(t)$ the distribution functions of $X-Y$ and $Y-X$, respectively.  Now, given that $\left\{ X_i-Y_i\right\}_{i=1}^n$ and $\left\{ Y_i-X_i\right\}_{i=1}^n$ are random samples of $X-Y$ and $Y-X$, respectively, the corresponding empirical distributions, $F_n$ and $G_n$,  can be used to both estimate their distribution functions  and propose a test statistic for the previous hypothesis. Under these notations, the empirical survival functions of $X-Y$ and $Y-X$ are given by
\[
\overline F_n(t)=1 - F_n(t)=\sum_{i=1}^n \frac{I_{(t,+\infty)}(X_i - Y_i)}{n} 
\]
and 
\[
\overline G_n(t)= 1 - G_n(t)=\sum_{i=1}^n \frac{I_{(t,+\infty)}(Y_i - X_i)}{n} ,
\]
respectively, where $I_{(t,+\infty)}(\cdot)$ is the indicator function on the interval $(t,+\infty)$. According to this, a reasonable Kolmogorov-Smirnov type test statistic is the following
\[
S_n^{st:wj}= \sqrt n \sup_{t\in[0,+\infty)} \left\{ \overline F_n(t) - \overline G_n(t) \right\},
\]
where the decision rule is to reject $H_0$ if $S_n^{st:wj}>c$ and the critical value $c$ would be determined in terms of the distribution of $S_n^{st:wj}$. However, it is not feasible to obtain the exact distribution of such statistic, and we would rather use the asymptotic distribution. Next, we provide an asymptotic upper bound for $P(S_n^{st:wj}>c)$. 

\begin{Pro} Following the previous notation and under $H_0$, we have that 
\[
\lim_{n\rightarrow \infty}P\left(S_n^{st:wj}>c\right) \le P\left(\sup_{t\in [0,+\infty)}\{GP(t)\}>c\right),
\]
for any $c\in \mathbb R$, being $GP=\{GP(t), t\in[0,+\infty)\}$ a Gaussian process where, for given $t_1,t_2,\ldots,t_k\in [0,+\infty)$, $(GP(t_1),GP(t_2),\ldots,GP(t_k))$ is a $k$-dimensional multivariate normal random vector with zero mean vector and covariance matrix $\Sigma$, such that
\[
\Sigma_{ii}=Var[GP(t_i)]=\overline F(t_i) + \overline G(t_i) - \left( \overline F(t_i) - \overline G(t_i) \right)^2, \text{ for }i=1,2,\ldots,k,
\]
and 
\[
\Sigma_{ij}=Cov(GP(t_i), GP(t_j))=\overline F(t_i \vee t_j) +  \overline G(t_i \vee t_j ) -(\overline F(t_i) - \overline G(t_i))(\overline F(t_j) - \overline G(t_j)),
\]
where $x \vee y=\max\{x,y\}$.
\end{Pro}

\begin{proof} First, we consider the weak convergence of the process $\Delta_n=\{\Delta_n(t);t\in[0,+\infty)\}$, where $\Delta_n(t)=\sqrt n \left(\overline F_n(t) - \overline G_n(t) - (\overline F(t) - \overline G(t))\right)$, which is centered at zero.  The stochastic process $\Delta_n$ has trajectories in the Skorokhod space $D([0,+\infty))$ of cadlag real valued functions with domain $[0,+\infty)$. It is well known that the processes $E_n(F)=\{F_n(t);t\in[0,+\infty)\}$ and $E_n(G)=\{G_n(t);t\in[0,+\infty)\}$ converge weakly, in the metric space $D([0,+\infty))$ with the uniform metric, to Brownian processes $B$ and $B'$, respectively. Then, from Theorem 1.4.8 by van der Vaart and Wellner (1996), we have that $(E_n(F), E_n(G))$ converges weakly to $(B, B')$, and by the continuity mapping theorem (CMT) we have that $\Delta_n=E_n(G) - E_n(F)$ converges weakly to $GP=B-B'$.  Let us identify now the limiting stochastic process $GP$. First, we observe that, 
\[
\overline F_n(t) - \overline G_n(t) = \sum_{i=1}^n \frac{l_t(X_i,Y_i)}{n},
\]
where $l_t(x,y)= I_{(t,+\infty)}(x-y) - I_{(t,+\infty)}(y-x)$ and the distribution of the random variable $l_t(X,Y)$ is given by
\begin{center}
$\left\{ \begin{tabular}{l c l}
1 & with probability & $\overline F (t)$, \\ 
0 & \text{with probability} & $G(t) - \overline F(t)$, \\ 
-1 & \text{with probability} & $\overline G(t)$. \\ 
\end{tabular} \right.$
\end{center}

Therefore, 
\[
E[l_t(X,Y)]= \overline F(t) - \overline G(t),
\]
\begin{equation}\label{eq-var}
Var[l_t(X,Y))]=\overline F(t) + \overline G(t) - \left( \overline F(t) - \overline G(t) \right)^2
\end{equation}
and 
\begin{equation}\label{eq-covar}
Cov(l_t(X,Y), l_{t'}(X,Y))=\overline F(t \vee t') +  \overline G(t \vee t' ) -(\overline F(t) - \overline G(t))(\overline F(t') - \overline G(t')),
\end{equation}
for any $t,t' \in[0,+\infty)$. Now, applying the multivariate central limit theorem (MCLT) to 
\[
\Delta_n(t)= \sqrt{n} \left(\frac{\sum_{i=1}^n l_t(X_i,Y_i)}{n} - E[l_t(X,Y)]\right),
\]
we have that $(\Delta_n(t_1), \Delta_n(t_2), \ldots,$ $\Delta_n(t_k))$ converges in distribution (weakly) to a multivariate normal random vector, for any $t_1, t_2,\ldots,t_k \in [0,+\infty)$, with zero mean vector and covariance matrix $\mathbf \Sigma$ given by  $\mathbf \Sigma_{ii}=\overline F(t_i) + \overline G(t_i) - \left( \overline F(t_i) - \overline G(t_i) \right)^2$ and $\mathbf \Sigma_{ij}= \overline F(t_i \vee t_j) +  \overline G(t_i \vee t_J) -(\overline F(t_i) - \overline G(t_i))(\overline F(t_j) - \overline G(t_j))$, for all $i,j=1,\ldots,n$, where the previous values follow from \eqref{eq-var} and \eqref{eq-covar}, respectively. 

Therefore, the limiting stochastic process is a Gaussian stochastic process. 

Now, given that the function $\sup(\cdot)$  is a continuous function on the Skorokhod space $D([0,+\infty))$ with the uniform metric, and using again the CMT, we have that  $\sup(\Delta_n)$ converges weakly (in distribution) to $\sup(GP)$.

Finally, under $H_0$, we have that $S_n^{st:wj} \le  \sup_{[0,+\infty)}(\Delta_n(t))$ (a.s), and therefore,
\[
\lim_{n\rightarrow \infty}P\left(S_n^{st:wj}>c\right) \le \lim_{n\rightarrow \infty}P\left(\sup_{t\in [0,+\infty)}\left\{\Delta_n(t)\right\}>c\right) = P\left(\sup_{t\in [0,+\infty)}\{GP(t)\}>c\right).
\]
\end{proof}

Now, we prove the consistency of the proposed test. 

\begin{Pro} Following the previous notation and under $H_1$, we have that 
\[
\lim_{n\rightarrow +\infty} P(S_n^{st:wj} > c)=1,
\]
for any $c\in \mathbb R$. 
\end{Pro}

\begin{proof} First we observe that under $H_1$ there exits a $t_0$ such that $\overline F(t_0) - \overline G(t_0) >0$, and therefore 
\[
 \sup_{t\in[0,+\infty)} \left\{ \overline F(t) - \overline G(t) \right\} >0.
\]

It is not difficult to see that 
\[
 \lim_{n\rightarrow +\infty} \sup_{t\in[0,+\infty)} \left\{ \overline F_n(t) - \overline G_n(t) \right\} =  \sup_{t\in[0,+\infty)} \left\{ \overline F(t) - \overline G(t) \right\} >0, \text{ a.s.}, 
 \]
therefore 
\[
 \lim_{n\rightarrow +\infty} \sqrt{n} \sup_{t\in[0,+\infty)} \left\{ \overline F_n(t) - \overline G_n(t) \right\} =  +\infty \text{ a.s.}. 
 \]
 
 Now the result follows observing that 
\[
\lim_{n\rightarrow +\infty} \inf P(S_n^{st:wj} > c) \ge P(\lim_{n\rightarrow +\infty} \inf \{S_n^{st:wj} > c\}) =1,
\]
for any $c\in \mathbb R$.

\end{proof}

In order to compute the upper bound of the $p$-value for the test, we need to compute a probability for the supreme of the limiting Gaussian process described in the previous result. Recall that the covariance matrix is unknown, hence we propose to replace the covariance by the consistent estimation provided by the sample covariance, that is, the values are obtained replacing the distribution functions $F$ and $G$ by their corresponding empirical distribution functions $F_n$ and $G_n$. Now, following Hansen (1996) and Barret and Donald (2003), we can use Monte-Carlo methods to provide simulations of the Gaussian process and then use a grid to approximate the probability. More specifically, select a grid $0=t_0<t_1<\cdots<t_k$ on $[0,+\infty)$, then provide a set of $N$ simulations of the multivariate random vector $(GP(t_0),\ldots,GP(t_k))$ and, finally, compute the upper bound of the $p$-value based on the maximum of each simulation as follows. Let $S_1,S_2,\cdots,S_N$ denote the maximum of each one of the simulations, then the upper bound can be approximated by 
\[
P\left(\sup_{t\in [0,+\infty)}\{GP(t)\}>c\right) \approx \frac{\sum_{i=1}^N I_{(c,+\infty)}S_i}{N}.
\]

The number of grid points can be defined by the user to increase the accuracy of the previous approximation. For the simulation of the multivariate random vector we have used \texttt{R} 3.6.3 (\texttt{R} Core Team, 2020) and the \texttt{MASS} (v7.3-51.51; Ripley et al., 2019) package. Recall that the simulation heavily depends on the computations related with the covariance matrix, which is time consuming. An alternative to this approach is the following. Given the grid $0=t_0<t_1<\cdots<t_k$ on $[0,+\infty)$, $P\left(\sup_{t\in [0,+\infty)}\{GP(t)\}>c\right)$ is approximated by  $P(\max(GP(t_0),\ldots,GP(t_k))>c)$ which in turn is approximated in terms of the empirical distribution based on the simulations. This probability can be computed as 
\[
P(\max(GP(t_0),\ldots,GP(t_k))>c)=1 - P(GP(t_0)\le c,\ldots,GP(t_k))\le c),
\]
that is, in terms of the multivariate distribution function of $(GP(t_0),\ldots,GP(t_k))$ at the point $(c,\ldots,c)$. This can be computed using  \texttt{R} 3.6.3 (\texttt{R} Core Team, 2020) and the \texttt{mvtnorm} (v1.1-0; Genz et al., 2020) package, which requires less computations for the covariance matrix and, therefore, it is less time consuming. 

To sum up, in order to test if $X\le_{\textup{st:wj}}Y$  for two dependent random variables with continuous distributions and  finite right extreme of the supports $T$, we propose to test
\[
H_0: P(X-Y>t) \le P(Y-X>t),\text{ for all }t\in[0,+\infty)
\]
against the alternative 
\[
H_1: P(X-Y>t) > P(Y-X>t),\text{ for some }t\in[0,+\infty),
\]
\noindent with the test statistic $\sqrt n \sup_{t\in[0,+\infty)} \left\{ \overline F_n(t) - \overline G_n(t) \right\}$ following the next steps: 
\begin{itemize}

\item[\textbf{S1.}] Take a bivariate random sample $\{(x_i,y_i)\}_{i=1}^n$ from $(X,Y)$.

\item[\textbf{S2.}] Compute the value $S_n^{st:wj}= \sqrt n \sup_{t\in[0,+\infty)} \left\{ \overline F_n(t) - \overline G_n(t) \right\}$, where $F_n$ and $G_n$ are the empirical distributions corresponding to the values $\{x_i-y_i\}_{i=1}^n$ and $\{y_i-x_i\}_{i=1}^n$, respectively. Is not difficult to see that the $S_n^{st:wj}= \sqrt n \max_{z\in Z\cup \{0\}} \left\{ \overline F_n(z) - \overline G_n(z
) \right\}$, where $Z=\{x_i-y_i\}_{i=1}^n \cup \{y_i-x_i\}_{i=1}^n$. 

\item[\textbf{S3.}] Fix a grid $0=t_0<t_1<\cdots<t_k$ on $[0,\max\{Z\}]$, consider the multivariate random vector $(GP(t_0),\ldots, GP(t_k))$, with zero mean vector and covariance matrix given by $\mathbf \Sigma^n$, where $\mathbf \Sigma_{ii}^n=\overline F_n(t_i) + \overline G_n(t_i) - \left( \overline F_n(t_i) - \overline G_n(t_i) \right)^2$ and $\mathbf \Sigma_{ij}^n= \overline F_n(t_i \vee t_j) +  \overline G_n(t_i \wedge t_J) -(\overline F_n(t_i) - \overline G_n(t_i))(\overline F_n(t_j) - \overline G_n(t_j))$, and then provide an approximation of the upper bound for the $p$-value by either: 

a) providing a set of $N$ simulations of the random vector $(GP(t_0),\ldots,GP(t_k))$, and approximating the upper bound $p_1 =  \sum_{i=1}^N I_{(S_n^{st:wj},+\infty)}S_i/N$, where $S_1,S_2,\cdots,S_N$ denote the maximum for each one of the simulations, 

or 

b) the probability $p_2 = 1 - P(GP(t_0)\le c,\ldots,GP(t_k))\le c)$.

\item[\textbf{S4.}] Reject the null hypothesis if either $p_1$ or $p_2$ is small enough. 

\end{itemize}

It is important to notice that the previous procedure can be used, with some small modifications, for the cases of integer-valued or ordinal random variables. If $X$ and $Y$ are integer-valued, then the supreme should be taken over the whole real line and the convergence results holds. If $X$ and $Y$ are integer-valued or ordinal random variables with a finite number of possible values then $Y-X$ takes values on a finite set, let say, $\{t_0, t_1,\ldots,t_k\}$. In this case, the stochastic process $GP$ is just a $(k+1)$-dimensional random vector and there is no need of a grid for the approximation,  being the upper bound for the $p$-value given by $P(\max(GP(t_0),\ldots,GP(t_k))>c)=1 - P(GP(t_0)\le c,\ldots,GP(t_k))\le c)$, which can be done as described previously. To sum up, for discrete or ordinal random variables we should proceed with steps S1, S2,  S3.b and finish with S4, where in the case of non-finite integer-valued random variables the supreme is taken over the whole real line.

Next, we provide a Monte Carlo experiment to assess the behaviour of the test in several situations. 

\subsection{Monte Carlo results}

Along this subsection we perform some Monte Carlo experiments for small and large samples in order to show how our proposed test behaves in different situations. In addition, the new test is compared with some well known tests for paired data. On the one hand, recall that the new stochastic dominance criterion is equivalent to the comparison of the means, when the difference $Y-X$ is normally distributed (see Example \ref{belip}) and, therefore, the Student's t test for paired data is equivalent to test the st:wj criterion in such case. On the other hand, according to Remark \ref{2.4}, the new criterion also implies that the median of $Y-X$ is greater than 0 (or equivalently the median of $X-Y$ is smaller than 0) whenever $P(X=Y)=0$, and this comparison is tested by the WMW test. Therefore, it is natural to compare our test with the usual Student's t and WMW tests for paired data, where $X$ being smaller than $Y$ plays the role of the null hypothesis in both cases. 

The experiment is performed in several situations where the weak joint stochastic dominance either holds or does not hold. Let us describe the different cases that we have considered.

\textbf{Cases 1, 2 and 3:} In this cases $(X,Y)$ follows a bivariate normal distribution with mean vector $(2, 4)$, $(3,1)$ and $(2, 2.01)$ in Case 1, 2 and 3 (C1, C2 and C3), respectively, with covariance matrix 
\[ 
\mathbf V= 
\begin{pmatrix}
  2 & 1.5 \\
  1.5 & 1.5  \\
\end{pmatrix}
\] 

\textbf{Cases 4, 5, 6, 7 and 8:} In this cases the dependence structure is given by a Clayton copula with parameter 0.5, that is, $C(u,v)= (u^{-0.5} + v^{-0.5})^{-2}$, for all $(u,v)\in [0,1]^2$, and the marginal distributions are either Pareto or Weibull, denoted by $X \sim P(a, k)$ and $X \sim W(a, b)$,  that is, the distribution is given by either $F(x)= 1 -\left( \frac{k}{x} \right)^a$,  for all $x\ge k$, or $F(x)= 1 -\exp\left( -(x/b)^a\right)$,  for all $x\ge 0$, respectively. In particular, $X \sim P(2, 1)$ and $Y \sim P(1.5, 1)$ in Case 4 (C4), $X \sim P(5, 4)$ and $Y \sim P(1.5, 1)$ in Case 5 (C5), $X \sim W(6, 2)$ and $Y \sim W(1.5, 1.5)$ in Case 6 (C6), $X \sim W(0.75, 4)$ and $Y \sim W(0.25, 1.5)$ in Case 7 (C7) and $X \sim W(0.5, 2)$ and $Y \sim W(0.9, 1.5)$ in Case 8 (C8).

Let us now justify the choice of the previous cases.  We have taken some bivariate normal distributions because of the fact that testing $X\le_{\textup{st:wj}}Y$ is reduced to the comparison of the means (see Example \ref{belip}). In particular, the random variables are ordered in the sense $X\le_{\textup{st:wj}}Y$ in Case 1, whereas the random variables are ordered in the reverse sense in Case 2. As regarding Case 3, the order is verified in the sense $X\le_{\textup{st:wj}}Y$ but the difference between the survival functions, as well as the difference of the means, is very small, hence it is difficult to detect it. Regarding the non normal cases, for cases 4 and 5, the new criterion does and does not hold, respectively  (see Figure \ref{mc}).  In cases 6, 7 and 8 the order does not hold in any sense, and there is an increasing difficulty to reject the null hypothesis (see Figure \ref{mc}). In addition, for cases 6 and 7 we have that $E[X] > E[Y]$, so the Student's t tests should reject the null hypothesis, whereas in Case 8 we have that $E[X] < E[Y]$, so the Student's t test should not reject the null hypothesis. 

\begin{figure}
\centering
\includegraphics[width=\textwidth]{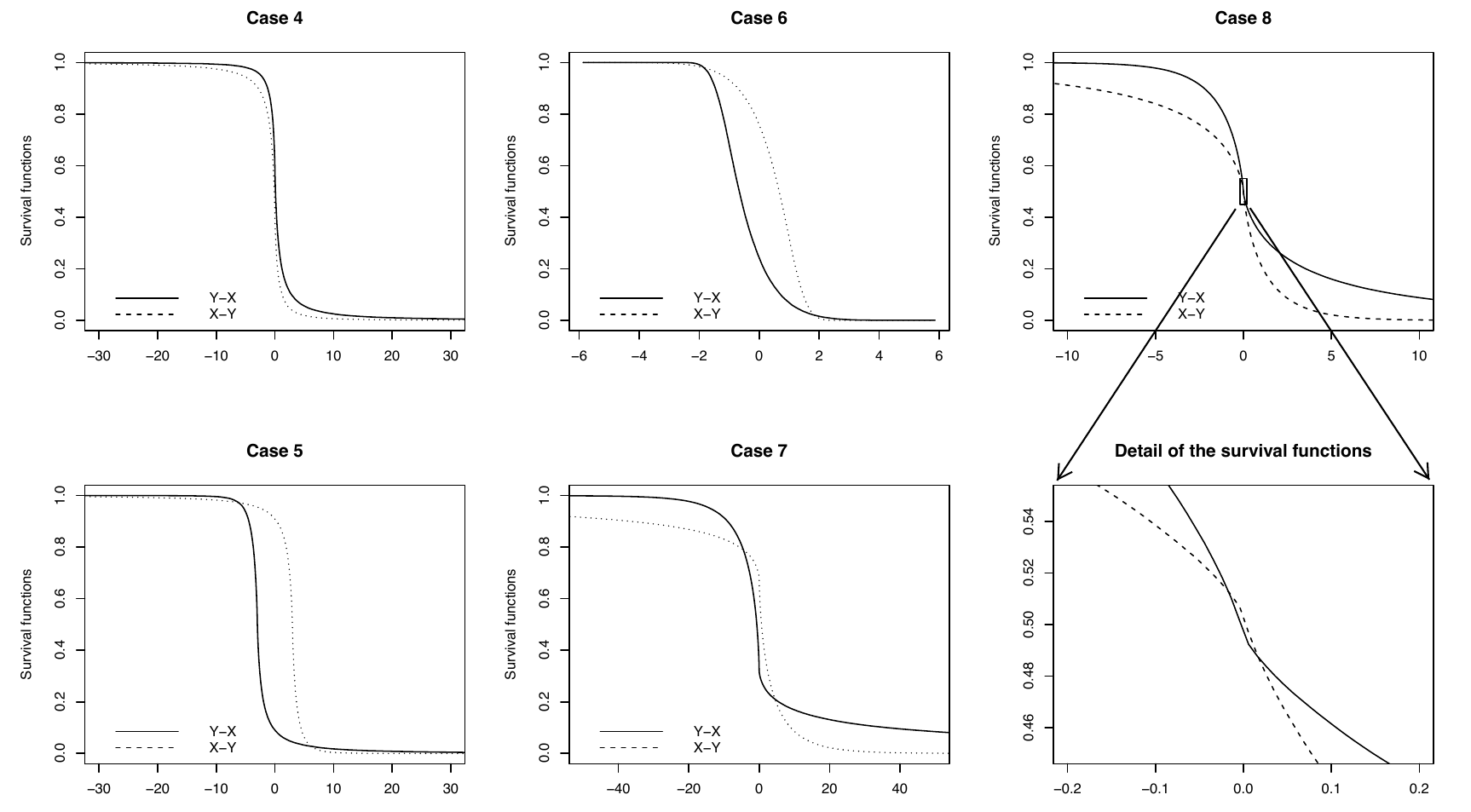}
\caption{\label{mc}\textit{Survival functions for $X-Y$ and $Y-X$ according to cases 4, 5, 6, 7 and 8.}}
\end{figure}

We have performed 1000 Monte Carlo replications for each case with different sample sizes ($n=50, 100, 200$ and $500$), in which the rejection rates of the null hypothesis have been computed for the two conventional significance levels of $0.05$ and $0.01$. The number of points of the grid is $k=100$ in every replication. The results are provided in Table I. Next, we make some observations on them. 

As we can see from Table I, when the order clearly holds (cases 1 and 4)  our test performs very well with a rejection rate of 0 per cent, even for small sample sizes in Case 1 whereas in Case 4 we need a sample of size 100 to get a 0.034 rate of rejection.  When the the null hypothesis clearly does not hold (cases 2, 5 and 6), the new test has a rejection rate of 100 per cent, even for small sample sizes. Therefore, the power of our test is very high in such situations.

Case 3 is a particular scenario where the null hypothesis holds but it is difficult to detect it. In this case, the rejection rate decreases as the sample size increases, being the rejection rates similar to those for the Student's t and WMW tests for large sample sizes. Therefore, the new test has a high power when the sample size increases. 

To finish, the order does not hold in cases 7 and 8 although it is difficult to reject the null hypothesis, hence we should expect low rejection rates. However, surprisingly, we get very high rejection rates, so the new test performs very well, in contrast with the WMW test. Regarding the Student's t test, we have a very low rejection rate for case 7 and an increasing rejection rate for case 8. Therefore, the new test detects more clearly that the random variable $Y$ should not be considered greater than the random variable $X$.

To sum up, the new test exhibits a very good behavior when the two random variables are either ordered or not ordered, even in the most difficult cases, where the behavior clearly improves as the sample size increases. 

\begin{table}
\scalebox{0.7}{\begin{tabular}{ccccccccccc}
 \multicolumn{10}{c}{Rejection rates for $\alpha=0.05$ ($\alpha=0.01$)}  \\ \hline
$n$ & Test & C1	& C2 & C3 & C4 & C5 & C6 & C7 & C8 \\ \hline 
& New st:wj	& 0 (0) & 1 (1) & 0.611 (0.520) & 0.160 (0.098) & 1.000 (1.000) & 1.000 (1.000) & 0.985 (0.982) & 0.997 (0.984) \\
$50$ & WMW	& 0 (0) & 1 (1) & 0.047 (0.008)  & 0.001 (0.000) & 1.000 (0.998) & 0.988 (0.947) &  0.258 (0.092) & 0.308 (0.103) \\ 
 & Student's t & 0  (0) & 1 (1) & 0.046 (0.008)  & 0.002 (0.000)  & 0.829 (0.763) & 0.975 (0.910) & 0.002 (0.001) & 0.768 (0.324) \\ \hline
& New st:wj	& 0 (0) & 1 (1) & 0.454 (0.336) & 0.034 (0.017) & 1.000 (1.000) & 1.000 (1.000) & 0.999  (0.997) & 0.999 (0.996) \\ 
$100$ & WMW	& 0 (0) & 1 (1) & 0.041 (0.012)  & 0.000 (0.000) & 1.000 (1.000) & 1.000 (1.000) & 0.391  (0.160) & 0.513 (0.241) \\ 
& Student's t & 0 (0) & 1 (1) & 0.040 (0.012) & 0.000 (0.000) & 0.863 (0.805) & 1.000 (0.996) & 0.000 (0.000) & 0.982 (0.805) \\ \hline 
& New st:wj	& 0 (0) & 1 (1) & 0.287 (0.164)  & 0.000 (0.000) & 1.000 (1.000) & 1.000 (1.000) & 1.000 (1.000) & 1.000 (0.998) \\
$200$ & WMW	& 0 (0) & 1 (1) & 0.032 (0.003)  & 0.000 (0.000) & 1.000 (1.000) & 1.000 (1.000) & 0.547  (0.329) & 0.769 (0.491) \\ 
& Student's t & 0 (0) & 1 (1) & 0.030 (0.004) & 0.000 (0.000) & 0.904 (0.861) & 1.000 (1.000) & 0.000 (0.000) & 0.999 (0.992) \\ \hline 
& New st:wj	& 0 (0) & 1 (1) & 0.083 (0.025) & 0.000 (0.000) & 1.000 (1.000)  & 1.000 (1.000) & 1.000 (1.000) & 1.000 (1.000) \\ 
$500$ & WMW	& 0 (0) & 1 (1) & 0.026 (0.003) & 0.000 (0.000) & 1.000 (1.000) & 1.000 (1.000) & 0.864 (0.668) & 0.979 (0.902)  \\ 
& Student's t & 0 (0) & 1 (1) & 0.021 (0.002) & 0.000 (0.000) & 0.931 (0.893) & 1.000 (1.000) & 0.000 (0.000) & 0.999 (1.000)  \\ \hline 
\end{tabular}}
\caption{Rejection rates for the null hypothesis. \label{tab1}}
\end{table}

\section{An application in finance}

Along this section we present an application of the new stochastic dominance criterion in the context of return data and portfolio selection.   The use of stochastic dominance criteria to select a portfolio is a very common practice in finance.   Given an investment portfolio, a natural question is how we can replace an asset of the given portfolio to provide a higher return.  It is well known that the usual stochastic dominance criterion does not provide an answer by two reasons. On the one side, it is difficult to find empirical evidence of stochastic dominance among two asset returns and, on the other side, replacing an asset by another one which stochastically dominates the first one does not result in a higher return for the new portfolio.  

Next, we explore the role of the new criterion in this context. Let us consider the comparison of some weekly return data according to the new criterion.  In order to eliminate the time dependent effect, we consider weekly returns.

\begin{Ex}\label{ex_app} We have considered 331 weekly return data for Twitter ($X$) and Facebook ($Y$) from January 1, 2018.  The data have been retrieved from \url{http://finance.yahoo.com} using \texttt{R} 3.6.3 (\texttt{R} Core Team, 2020) and the \texttt{quantmod} (v0.4.17; Ryan et al., 2020) package. Figure \ref{app1_FBT} shows that $X$ and $Y$ are highly positive correlated. Now, the question is the following. Can we state that one of them tends to take larger values than the other one in some probabilistic sense? Figure \ref{app1_FBT} shows that the usual stochastic dominance does not hold in any sense (the Q-Q plot clearly intersects the diagonal $x=y$) and, therefore, this criterion does not provide any answer.
 \begin{figure}
 \centering
\includegraphics[width=\textwidth]{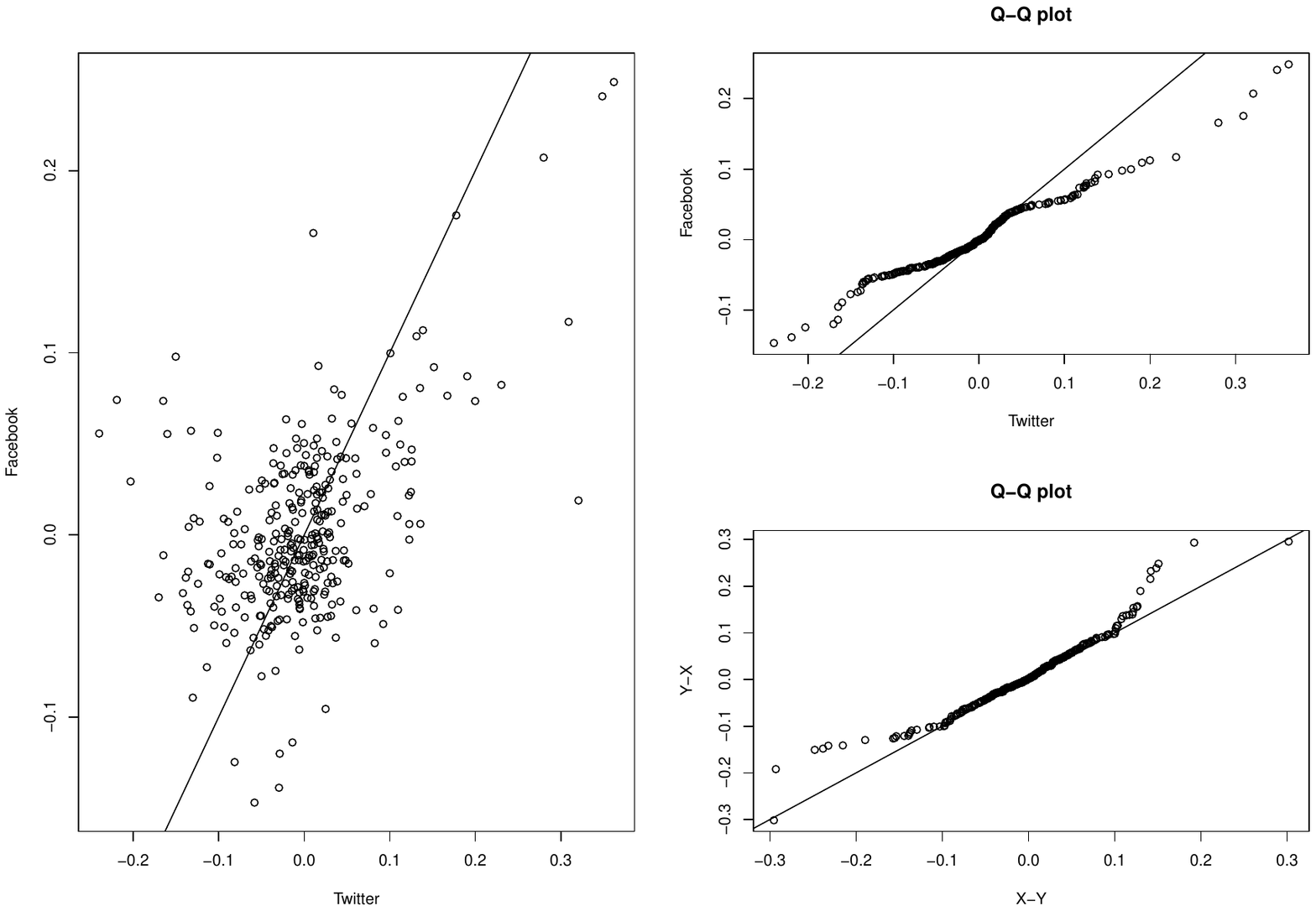}
\caption{\label{app1_FBT}\textit{Scatter plot for Twitter and Facebook weekly returns (on left side), Q-Q plot for Twitter and Facebook returns (on the right side, above) and Q-Q plot for the difference between Twitter (X) and Facebook (Y) returns (on the right side, below).}}
\end{figure}

This is a common example where the usual stochastic dominance does not hold.  However,  the Q-Q plot of $Y-X$ and $X-Y$ suggests that $X\le_{\textup{st:wj}}Y$ (see Figure \ref{app1_FBT}) and, therefore, Facebook weekly returns seem to take larger values than the Twitter ones. 

In order to confirm the former hypothesis, we apply the test provided in Section 3.  In particular, we test $H_0: X\le_{\textup{st:wj}}Y$ against the alternative hypothesis $H_1: X \nleq_{\textup{st:wj}}Y$. For this case we get a $p$-value=0.9167104 and, therefore,  there is no empirical evidence against $X\le_{\textup{st:wj}}Y$. However, this not the end of our analysis. We need to check if the previous order holds strictly. From Remark \ref{2.4} we have that if $X\le_{\textup{st:wj}}Y$ and $E[X]=E[Y]$, then $X=_{\textup{st:wj}}Y$. Therefore, we only need to check if the means are strictly ordered or not. In this case the Student's t test for the comparison of the means, in the case of paired data, gives a $p$-value=0.2142, and therefore there is not enough empirical evidence against the hypothesis $X=_{\textup{st:wj}}Y$.

Let us consider another case in which we consider weekly return data for Amazon (X) and Facebook (Y) over the same time period.

As we can see in Figure \ref{app2_FBT} $X$ and $Y$ are highly positive correlated, the usual stochastic dominance does not hold in any sense but the new stochastic dominance criterion seems to hold. Again we apply the test provided in Section 3 to test $H_0: X\le_{\textup{st:wj}}Y$ against the alternative hypothesis $
H_1: X \nleq_{\textup{st:wj}}Y$. In this case we get a $p$-value=0.9946772 and, therefore,  there is no empirical evidence against $X\le_{\textup{st:wj}}Y$.  Furthermore, in this case the Student's t test for the comparison of the means gives a $p$-value=0.008786, and therefore the means are different and $X\le_{\textup{st:wj}}Y$ strictly.

\begin{figure}
 \centering
\includegraphics[width=\textwidth]{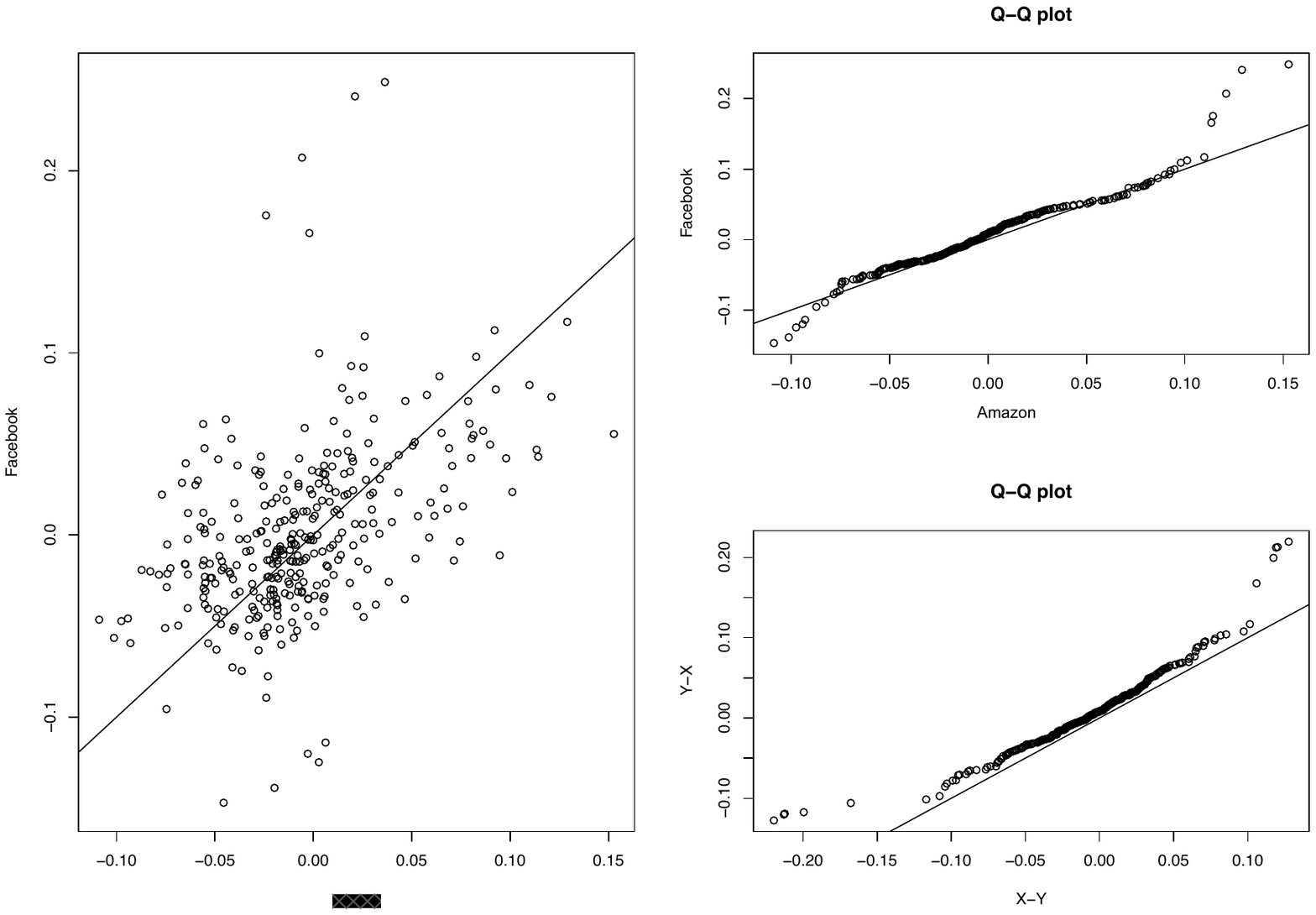}
\caption{\label{app2_FBT}\textit{Scatter plot for Amazon and Facebook weekly returns (on left side), Q-Q plot for Amazon and Facebook returns (on the right side, above) and Q-Q plot for the difference between Amazon (X) and Facebook (Y) returns (on the right side, below).}}
\end{figure} 

\end{Ex}

This fact can be used and exploited in the context of portfolio selection. Let us consider two portfolios where one of them is obtained from the other one by replacing one of the assets.  In particular, we have two portfolios with returns $P_1 = (1-\alpha) X + \alpha Z$ and $P_2  = (1-\alpha)Y + \alpha Z$. Let us assume that $X\le_{\textup{st:wj}}Y$, and let us compare $P_1$ and $P_2$ according to the new criterion. Given that $P_1 - P_2=(1-\alpha)(X - Y)$, $P_2 - P_1=(1-\alpha)(Y-X)$ and the  preservation of the usual stochastic dominance criterion under increasing transformations, we have that $P_1\le_{\textup{st:wj}} P_2$. Therefore, the portfolio $P_2$ has a greater return than the portfolio $P_1$ according to the new criterion, as we illustrate next in Example 4.2. This is not always true if the new criterion is replaced, in the previous discussion,  by the usual stochastic dominance. Let us see an example of this situation using the previous data set. 

\begin{Ex}
Continuing with Example \ref{ex_app}, let us consider the weekly returns for Twitter ($Z$), Amazon ($X$) and Facebook ($Y$), and let us consider the portfolios $P_{1} = 0.8 X + 0.2 Z$ and $P_2  = 0.8 Y + 0.2 Z$.  According to the previous discussion and the conclusions in the previous example, we have that $P_1 \le_{\textup{st:wj}}P_2 $. It is interesting to notice that the usual stochastic dominance criterion does not hold among $P_1$ and $P_2$ (see Figure \ref{app3_FBT} on the left)  but the new stochastic dominance criterion does (see Figure \ref{app3_FBT} on the right).

\begin{figure}
\centering
\includegraphics[width=\textwidth]{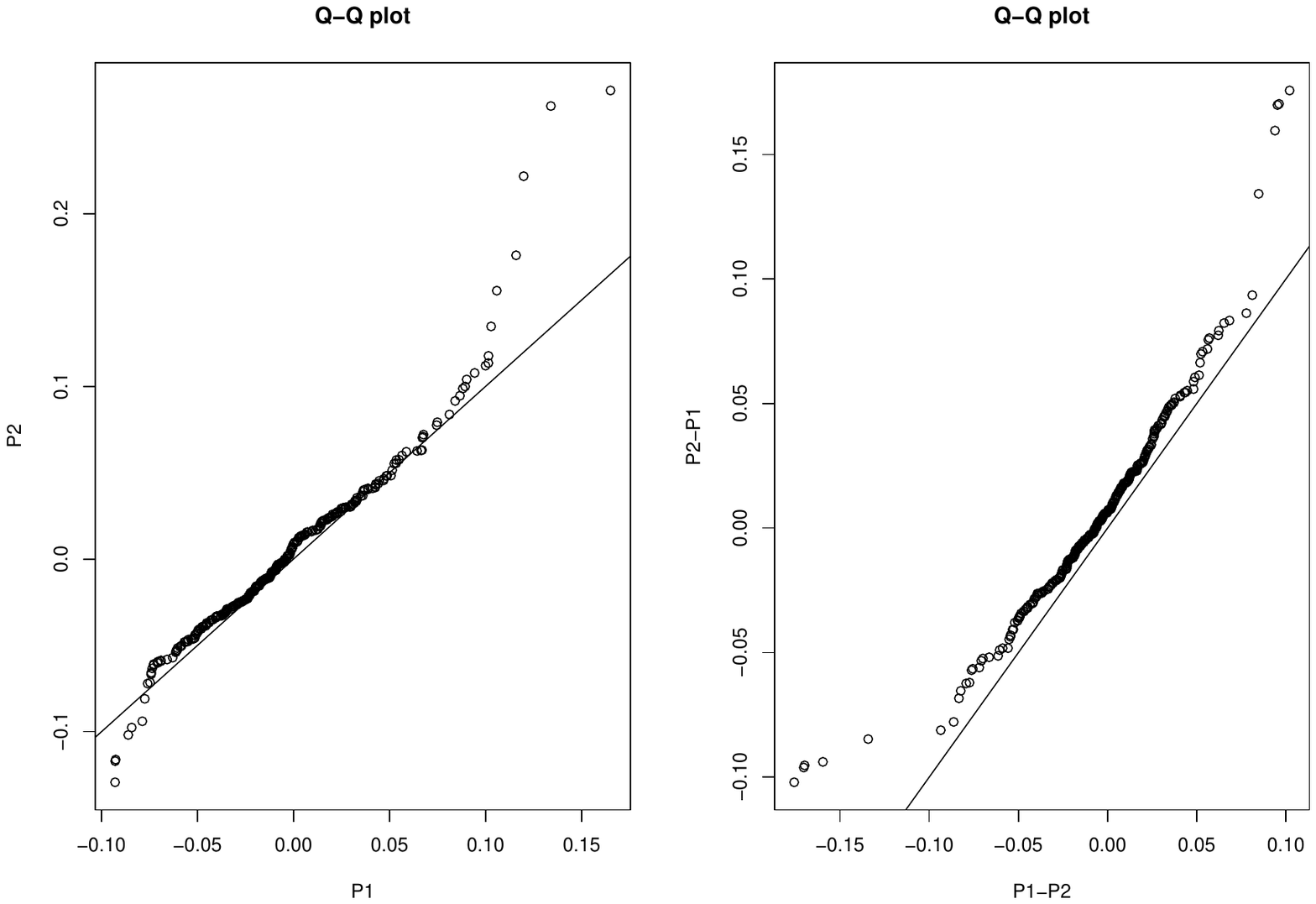}
\caption{\label{app3_FBT}\textit{Q-Q plot for the portfolios $P_1$ and $P_2$ (on the left side) and Q-Q plot for the differences of the two portfolios (on the right side).}}
\end{figure} 
\end{Ex}

Let us give another application. Let us consider a pair $(X,Y)$ of non independent random returns and the portfolios $P_1= (1-\alpha_1) X + \alpha_1 Y$ and $P_2 = (1-\alpha_2) X + \alpha_2 Y$.  Given that  $P_2 - P_1 = (\alpha_2 - \alpha_1) (Y-X)$, $P_1 - P_2 = (\alpha_2 - \alpha_1) (X-Y)$ and the preservation of the usual stochastic dominance under increasing transformations, we have that $P_1 \le_{\textup{st:wj}}P_2$, if $X\le_{\textup{st:wj}}Y$ and $\alpha_1 < \alpha_2$. Therefore, putting more weight to the greatest asset gives a greater return, according to the new criterion.

\section{Conclusions and future research}

In this paper we have introduced a new stochastic dominance criterion which takes into account the dependence structure of the random variables involved in the comparison, as well as it has been used to compare paired data. The new criterion has been shown to be coherent with the Student's t and WMW tests for paired samples, but it provides a more informative comparison of dependent random variables. The new criterion has been studied from the probabilistic and inferential point of view, providing a new test to check it. In addition, we have given an application in finance which shows how this new criterion can be fruitfully used in portfolio selection. It is also clear that this new test can be used in several other contexts, like clinical trials, where ``before" and ``after" data is a common situation. This study can be enlarged and enriched according to the following items, among others 

\textbf{1. Further properties of the st:wj criterion:} As for the preservation under increasing transformations is concerned, the weak joint stochastic dominance is not preserved by increasing transformations of the marginal distributions. That is, if $X \le_{\textup{st:wj}} Y$, then $\phi(X) \nleq_{\textup{st:wj}} \phi(Y)$, for every increasing function $\phi$. The argument is the following. Assume that $\phi(X) \leq_{\textup{st:wj}} \phi(Y)$, for any increasing function $\phi$. Then, by Remark \ref{2.4}, $E[\phi(X)]\leq E[\phi(Y)]$, for any increasing function $\phi$, provided that the previous expectations exist, and, consequently, $X \leq_{\textup{st}}Y$. However, we have seen in Section 2  that the weak joint stochastic dominance does not imply the usual stochastic dominance. Therefore, it would interesting to find some other families of functions that preserve the new criterion under transformations of the marginals distributions.  

Additionally, if $X \le_{\textup{st:wj}} Y$ and $Y \le_{\textup{st:wj}} Z$ then it is not necessarily true that $X \le_{\textup{st:wj}} Z$. The main reason is that the comparisons $X \le_{\textup{st:wj}} Y$ and $Y \le_{\textup{st:wj}} Z$ does not contain any information about the dependence structure of $(X,Z)$. Therefore, finding conditions to ensure $X \le_{\textup{st:wj}} Z$ is worth considering. 

\textbf{2. Multivariate extensions:} A natural consideration is how to define a multivariate version of the new stochastic dominance criterion. Let us consider a $2n$ dimensional random vector $(\mathbf X, \mathbf Y)$, where $\mathbf X$ and $\mathbf Y$ are $n$ dimensional random vectors. A natural extension in the multivariate case is the following: 

We say that $\mathbf X$ is smaller than $\mathbf Y$ in the multivariate joint weak stochastic dominance, denoted by $\mathbf X \le_{\textup{st:wj}} \mathbf Y$, if $\mathbf X - \mathbf Y \le_{\textup{st}} \mathbf Y - \mathbf X$, where $\le_{\textup{st}}$ denotes the multivariate stochastic dominance (see Shaked and Shanthikumar, 2007 and Belzunce et al., 2016a). 

We are working on properties and inferential issues for this definition.

\section*{Acknowledgements} The authors want to acknowledge the comments by two anonymous referees which have improved significantly the presentation of this paper.  We sincerely thank Prof. Franco Pellerey, who send us some comments and examples on an earlier version of this paper. We also thank Prof. Tomasso Lando who pointed us to the paper by Montes et al. (2020).

Funding: This work was supported by the Ministerio de Ciencia e Innovaci\'on of Spain under grant PID2019-103971GB-I00/AEI/10.13039/501100011033.


\begin{thebibliography}{99}

\bibitem{} Andreoli, F. (2018). Robust inference for inverse stochastic dominance. \textsl{Journal of Business \& Economic Statistics}, \textbf{36},  146--159.

\bibitem{AKS02} Arcones, M.A., Kvam, P. and Smaniego, F.J. (2002).  Nonparametric estimation of a distribution subject to a stochastic precedence constraint. \textsl{Journal of the American Statistical Association}, \textbf{97}, 170--182.

\bibitem{} Barret, G.F. and Donald, S.G. (2003). Consistent tests for stochastic dominance. \textsl{Econometrica}, \textbf{71}, 71--104. 

\bibitem{} Barrett,  G.F., Donald, S.G.  and Bhattacharya, D. (2014). Consistent nonparametric tests for Lorenz dominance.  \textsl{Journal of Business \& Economic Statistics},  \textbf{32},  1--13.

\bibitem{} Bell, D. (1982). Regret in decision making under uncertainty. \textsl{Operations Research}, \textbf{30}, 961--981.

\bibitem{} Belzunce, F., Mart\'{i}nez-Riquelme, C. and Mulero, J.M. (2016a). \textsl{An Introduction to Stochastic Orders}. Elsevier-Academic Press, Amsterdan. 

\bibitem{} Belzunce, F., Mart\'{i}nez-Riquelme, C., Pellerey, F. and Zalzadeh, S. (2016b). Comparison of hazard rates for dependent random variables. \textsl{Statistics}, \textbf{50}, 630--648. 

\bibitem{BSC04} Boland, P.J., Singh, H. and Cukic, B. (2004). The stochastic precedence ordering with applications in sampling and testing. \textsl{Journal of Applied Probability}, \textbf{41}, 73--82.

\bibitem{} Cai, J., and Wei, W. (2014). Some new notions of dependence with applications in optimal allocation problems.  \textsl{Insurance:  Mathematics and  Economics},  \textbf{55}, 200--209.

\bibitem{} Davidson, R. and Duclos,  J.Y.  (2000).  Statistical inference for stochastic dominance and for the measurement of poverty and inequality. \textsl{Econometrica},  \textbf{68}, 1435--1464.

\bibitem{} Denuit, M., Dhaene, J., Goovaerts, M. and Kaas, R.  (2005). Actuarial theory for dependent risks. Wiley, Chichester.

\bibitem{} Divine, G.W., Norton, H.J., Baron, A.E. and Juarez-Colunga, E. (2018). The Wilcoxon-Mann-Whitney procedure fails as a test for medians. \textsl{The American Statistician}, \textbf{72}, 278--286.  

\bibitem{} Fishburn, P. (1982). Nontransitive measurable utility. \textsl{Journal of Mathematical Psychology}, \textbf{26}, 31--67.

\bibitem{} Genz. A., Bretz, F., Miwa, T., Mi, X.,  Leisch, F., Scheipl, F., Bornkamp, B., Maechler, M., Hothorn, T.,  (2020). \texttt{mvtnorm}: Multivariate Normal and t Distributions. \texttt{R} package   version 1.1-0.

\bibitem{} Hansen, B.E. (1996). When a nuisance parameter is not identified under the null hypothesis. \textsl{Econometrica}, \textbf{64}, 413--430. 

\bibitem{} Hon Tan, C. and Hartman, J. (2013). Stochastic dominance, regret dominance and regret theoretic dominance. \textsl{EURO Journal on Decision Processes}, \textbf{1}, 285--297.

\bibitem{} Lehmann, E.L. (1955). Ordered families of distributions. \textsl{The Annals of Mathematical Statistics}, \textbf{26}, 399--419.

\bibitem{} Levy, H. (2016). \textsl{Stochastic Dominance}. Springer, London. 

\bibitem{} Linton, O.,  Maasoumi, E.  and Whang, Y.J. (2005).  Consistent testing for
stochastic dominance under general sampling schemes.  \textsl{The Review of
Economic Studies},  \textbf{72,} 735--765. 

\bibitem{} Loomes, G. and Sugden, R. (1982). Regret theory: An alternative theory of rational choice under uncertainty. \textsl{The Economic Journal}, \textbf{92}, 805--824.

\bibitem{} Montes, I.,  Salamanca, J.J. and Montes, S. (2020). A modified version of stochastic dominance involving dependence. \textsl{Statistics and Probability Letters}, \textbf{165}, 1--12.

\bibitem{} Mulero, J., Sordo, M.A., de Souza, M.C. and Suárez-LLorens, A. (2017). Two stochastic dominance criteria based on tail comparisons. \textsl{Applied Stochastic Models in Business and Industry}, \textbf{33}, 575--589.

\bibitem{} M\"{u}ller, A. and Stoyan, D. (2002). \textsl{Comparison methods for stochastic models and risks}. Wiley Series in Probability and Statistics. John Wiley \& Sons, Ltd., Chichester.

\bibitem{} Navarro, J.  and  Sarabia, J. (2022). Copula representations for the sum of dependent risks: Models and comparisons.  \textsl{Probability in the Engineering and Informational Sciences},  \textbf{36},  320--340.

\bibitem{NE06} Nelsen, R. B. (2006). \textit{An Introduction to Copulas}, Springer-Verlag, New York.

\bibitem{} \texttt{R} Core Team (2020). \texttt{R}: A language and environment for statistical
  computing. \texttt{R} Foundation for Statistical Computing, Vienna, Austria.  URL http://www.R-project.org/.

\bibitem{} Ripley, B., Venables, B., Bates, D.M., Hornik, K., Gebhardt, A. and Firth, D. (2019). \texttt{MASS}: Support Functions and Datasets for Venables and Ripley's MASS. \texttt{R} package   version 7.3-51.5.

\bibitem{}   Ryan, J.A., Ulrich, J.M., Thielen, W., Teetor, P. and Bronder, S. (2020). \texttt{quantmod}: Quantitative Financial Modelling Framework. \texttt{R} package   version 0.4.17.

\bibitem{} Scaillet, O.  and Topaloglou, N.  (2010).  Testing for stochastic
dominance efficiency. \textsl{Journal of Business \& Economic Statistics},  \textbf{28}, 169--180.

\bibitem{Sha} Shaked, M. and Shanthikumar, J.G. (2007). \textsl{Stochastic
Orders}, Springer-Verlag, New York.


\bibitem{SHA2}  Shanthikumar, J.G. and Yao, D.D. (1991). Bivariate characterization of some stochastic order relations. \textit{Advances in Applied Probability}, \textbf{23}, 642--659.

\bibitem{} Sriboonchitta, S., Wong, W.K., Dhompongsa, S. and Nguyen, H.T. (2010). \textsl{Stochastic Dominance and Applications to Finance, Risk and Economics}, CRC Press, Boca Raton, FL.

\bibitem{} van der Vaart, A. and Wellner, J.A. (1996). \textsl{Weak Convergence and Empirical Processes}. Springer, New York.

\end{thebibliography}
\end{document}